\documentclass[11pt]{article} 
\usepackage{graphicx} 
\usepackage{amsmath,amscd,amssymb, hyperref, amsthm, mathtools,cases,amsopn, soul,cleveref,mathabx,relsize,tikz}
\usepackage{minted,enumerate}
\newtheorem{theorem}{Theorem}[subsection]
\newtheorem{corollary}[theorem]{Corollary}
\newtheorem{observation}[theorem]{Observation}

\newtheorem{proposition}[theorem]{Proposition}

\theoremstyle{definition}
\newtheorem{definition}[theorem]{Definition}

\newtheorem{example}[theorem]{Example}

\definecolor{cardinal}{rgb}{0.77, 0.12, 0.23}

\newcommand{\Z}{\operatorname{Z}}
\DeclareMathOperator{\dist}{dist}

\DeclareMathOperator{\ab}{(\alpha,\beta)}
\DeclareMathOperator{\pt}{\mathrm{pt}}

\newcommand*{\ora}[1]{\overrightarrow{#1}}
\DeclareMathOperator{\Zab}{Z_{\alpha,\beta}}

\tikzstyle{vertex}=[circle, draw, inner sep=1pt, minimum size=8pt]

\def\GraphSix#1"{\begingroup(\itshape graph6:\nolinebreak[3]\ \verb"\aftergroup\endgroup\aftergroup)}

\usepackage{tikz}
\tikzstyle{vertex}=[circle, draw, inner sep=0pt, minimum size=10pt]
\tikzstyle{bvertex}=[circle, blue, fill, draw=black, inner sep=0pt, minimum size=10pt]
\tikzstyle{Bvertex}=[circle, black, fill, draw, inner sep=0pt, minimum size=10pt]
\newcommand{\vertex}{\node[vertex]}
\newcommand{\bvertex}{\node[bvertex]}

\usepackage{authblk}

\crefname{figure}{figure}{}
\Crefname{figure}{Figure}{}
\title{Transmission Zero Forcing}
\author[1]{Adam H.~Berliner}
\author[2]{Chassidy Bozeman} 
\author[3]{Karen L. Collins} 
\author[4]{Mary Flagg}
\author[5]{Veronika Furst} 
\author[6]{Mark Hunnell}
\affil[1]{Department of Mathematics, Statistics, and Computer Science, St.~Olaf College, Northfield, MN}
\affil[2]{ Department of Mathematics and Statistics, Mount Holyoke College, South Hadley, MA}
\affil[3]{Department of Mathematics and Computer Science, Wesleyan University, Middletown, CT}
\affil[4]{Department of Mathematics and Computer Science, University of St.~Thomas, Houston, TX}
\affil[5]{Department of Mathematics, Fort Lewis College, Durango, CO}
\affil[6]{Department of Mathematics, Winston-Salem State University, Winston-Salem, NC}

\begin{document}
\maketitle

\begin{abstract}
    We initiate the study of transmission zero forcing, a variant of the well-studied zero forcing graph parameter. In this variant, a subset of vertices is assigned an initial unit weight, and these vertices can increase the weight of a neighbor subject to the zero forcing color change rule at a rate determined by the transmission proportion. A vertex is considered filled when its weight exceeds the transmission threshold, at which point the process can continue. The transmission zero forcing number of a graph is the minimum cardinality of the initial set that results in all vertices exceeding the transmission threshold. This iterative graph coloring process is a generalization of zero forcing that allows for a vertex to be forced by multiple neighbors. We develop tools for studying this graph parameter, determine its value on some common classes of graphs, and investigate its behavior under various graph operations.
\end{abstract}

\section{Introduction}

In many interesting networks modeled by graphs, such as social media, the electric grid, and disease propagation, nodes can exert influence over one another subject to certain constraints.  Understanding the dynamics of the evolution of these networks is a difficult problem, hence various attempts to understand simpler models first.  Several natural graphical models with vertex labels changing over time have been studied, including zero forcing \cite{Barioli2010ZeroFP} and its related variants, domination \cite{domination1998} and its related variants, and cellular graph automata \cite{WU1979305}.  In this article, we introduce a novel variant of the defining mechanics of the zero forcing paradigm and study its implications in contrast to the current zero forcing literature.  We call this new variant transmission zero forcing, or transmission forcing for short.

We delay formal definitions until \Cref{prelim} but informally motivate our problem in this section.    
Zero forcing and its variants have generated significant interest because of a wide variety of applications, including linear algebra \cite{Barioli2010ZeroFP}, quantum control \cite{quantum}, and the spread of disease and information \cite{DREYER2009}. In this paradigm, each vertex is initially designated as filled or unfilled, and the graph evolves according to a color change rule specific to the variant.  Originally, the zero forcing literature used color assignments to designate filled and unfilled vertices, hence the use of the term color change rule.  The parameter of interest is the minimum cardinality of an initially filled set of vertices that results in all vertices filled after repeated application of the color change rule.  Transmission zero forcing refines this model by attaching weights to each vertex (zero or one initially) and two additional parameters which take values in the interval $[0,1]$.  The first parameter, called the transmission proportion $\alpha$, determines the proportion of weight a filled vertex may assign to an unfilled vertex, subject to the color change rule.  The second parameter, which we call the transmission threshold proportion $\beta$, specifies the minimum weight a vertex must have to be filled.  Thus standard zero forcing corresponds to transmission zero forcing when $\alpha=1$ and $\beta$ is arbitrary. Moreover, in standard zero forcing each vertex can receive weight at most once and transmit weight at most once.  Our variant relaxes the former restriction while maintaining the latter. 

We now describe some of the graph evolution variants well-studied in the literature to clarify the gap which we address.
Zero forcing, as developed in the literature, designates a single filled vertex to influence an unfilled vertex, and each filled vertex can influence at most one unfilled vertex as the graph evolves.  Several natural variations, including PSD forcing \cite{Barioli2010ZeroFP} and $k$-forcing \cite{AMOS20151}, allow a filled vertex to influence more than one unfilled vertex.  PSD forcing was introduced in connection to the problem of determining the maximum nullity achieved by a positive semidefinite matrix that obeys the combinatorial pattern defined by a graph. If $B \subseteq V(G)$, where $V(G)$ denotes the set of vertices of a graph $G$, is a set of filled vertices, then for PSD forcing $v\in B$ can fill a vertex in each component of the induced subgraph of $G$ obtained by deleting $B$, subject to the PSD color change rule.  On the other hand, $k$-forcing was motivated by graph-theoretic questions and allows each vertex in $B$ to fill up to $k$ of its neighbors, subject to the $k$-forcing color change rule.  Both variants, however, designate a single filled vertex to perform the force that changes an unfilled vertex to filled.  Transmission forcing inverts these restrictions, allowing an unfilled vertex to receive contributions to its weight from all of its (filled) neighbors, but a filled vertex may only contribute to the weight of at one most one of its (unfilled) neighbors.  

Transmission forcing shares some similarities to $(t,r)$ broadcast domination, first introduced in \cite{BLESSING2015}.  In this domination variant, one seeks to minimize the cardinality of a set $B \subseteq V(G)$ such that each $v\in B$ distributes a weight of $\max\{0,t- d(u,v)\}$ to the vertices $u \in V(G)$, where $d(u,v)$ denotes the distance between $u$ and $v$ in $G$.  Such a set is $(t,r)$ dominating if each vertex of $G$ has a weight of at least $r$.  Thus our transmission threshold $\beta$ plays the role of $r$ in $(t,r)$ domination, and the decay over distance in transmission zero forcing is geometric rather than linear as it is in $(t,r)$ domination. 

Before concluding this section with a description of the organization of the article, we describe a model in which transmission forcing promotes an efficient resource allocation strategy.  We model a transportation system of an urban network by a graph, with one vertex in the graph for each city, and an edge between two cities if they are adjacent stops on a trucking route. Our goal is to identify the smallest possible number of warehouses, each with its own truck, so that the trucks may be sent out and will supply all of the cities with necessary materials. Each truck driver delivers a proportion $\alpha$ of their load, and each warehouse stockpiles the deliveries until it reaches an amount $B$ before sending out a truck. When a truck $T$ originates from its warehouse $W$, it may only depart if all but one of the neighboring cities to $W$ have been supplied. In that case, $T$ departs for the unique adjacent city $C$ that has not yet been supplied, and delivers a fraction $\alpha$ of its load to the warehouse in $C$, and keeps $(1-\alpha)$. Subsequently, whenever a truck arrives at a city, it adds $\alpha$ of its load to the city's warehouse, and when the warehouse is filled with at least $B$ supplies, the city loads a truck with this amount, and sends the truck out once the departure condition is met. The process continues until every city has been supplied. We assume that every initial warehouse has $A$ amount. We let $\beta=B/A$. The model promotes efficiency in its delivery method because when a truck departs for a city, it is the only destination that has not yet been supplied. 

\Cref{prelim} provides the definitions and notation for the graph theoretic notions needed in this article, before developing zero forcing and transmission zero forcing. \Cref{sec:transmission} establishes some basic properties of the transmission zero forcing number, including its extreme values, while \Cref{sec:families} establishes values of the transmission forcing number for several graph families.  Finally, \Cref{sec:vert_edge_del} describes the effect of graph minor operations on the transmission forcing number.

\section{Preliminaries}\label{prelim}

A (finite, simple) graph is denoted by $G = (V (G), E(G))$ where $V (G)$ is the set
of vertices of $G$ and $E(G)$ is the edge set of $G$, where an edge is a two-element set of vertices
and the edge $\{u, v\}$ may also be denoted by $uv$ (or $vu$).  
For $S \subseteq V (G)$, the subgraph of $G$ induced by $S$, denoted by $G[S]$, is the graph obtained
from $G$ by removing all vertices not in $S$ (and their incident edges).  We denote the path graph on $n$ vertices by $P_n$, the cycle graph on $n$ vertices by $C_n$, the complete graph on $n$ vertices by $K_n$, and the complete bipartite graph by $K_{p,q}$, where $p+q=n$.  Two vertices $u$ and $v$ are called neighbors, or adjacent, if $uv \in E(G)$, and the open neighborhood of $v$, denoted $N(v)$, is the set of vertices adjacent to $v$.  The closed neighborhood of $v$ is $N[v]=N(v) \cup \{v\}$. The degree of a vertex in $G$, denoted $\deg_G(v)$, is the number of neighbors of $v$ in $G$, i.e., $|N(v)|$. We denote by $\Delta(G)$ (respectively $\delta(G)$) the maximum (resp. minimum) degree of a vertex in $G$. A vertex of degree-1 is called a leaf.  Given $u,v \in V(G) $, the distance from $u$ to $v$, denoted $\dist_G(u,v)$, is the length of the shortest path from $u$ to $v$.  Similarly,  given a directed graph $\ora{G}$ and two vertices $u,v$ of $G$, the distance from $u$ to $v$ in $\ora{G}$, denoted $\ora{\dist}_{\scriptsize \ora{G}}(u,v)$, is the length of the shortest (directed) path from $u$ to $v$.  The diameter of $G$, denoted $\mathrm{diam}(G)$, is the maximum distance among every possible pair of vertices in $G$.  If $v\in V(G)$ and $e = uv \in E(G)$, then $G-v$ is the graph obtained from $G$ by deleting the vertex $v$, $G-e$ is the graph obtained from $G$ by deleting the edge $e$, and $G/e$ is the contraction of $G$, obtained from $G$ by identifying $u$ and $v$ and suppressing any loops or multiple edges that result from doing so.

\subsection{Zero Forcing}\label{subsec:zf}

Let $G=(V(G),E(G))$ be a graph and assume that each vertex is designated as either filled or unfilled. Let $B_0$ be the set of initially filled vertices.  The zero forcing process evolves according to the (standard) color change rule: If an unfilled vertex $u$ is the unique unfilled neighbor of $v \in B_0$, then $u$ becomes filled.  We also say that $v$ forces $u$ and denote this by $v\rightarrow u$.  After all possible forces are made, we iterate the color change rule using the new set of filled vertices until no more forces are possible.  The sets of filled and unfilled vertices determine a (unique) final coloring of $G$ from $B_0$, and $B_0$ is a called a {\em zero forcing set} if all the vertices of $G$ are filled in the final coloring.  The {\em zero forcing number} of $G$ is the minimum cardinality of a zero forcing set and is denoted by $\Z(G)$.

The {\em propagation time} of a graph $G$ is the minimum number of iterations of the color change rule necessary for a minimum zero forcing set to achieve a final coloring.  It is often convenient to have notation for the sets of filled and unfilled vertices at each iteration.  Let $G$ be a graph and $B \subseteq V(G)$ a set of initially filled vertices. For $t\geq 1$, define:
\begin{align*}
         B_0\;\;  &= B_{(0)} = B, \\
        B_{(t)} &= \{u \in V(G)\setminus B_{t-1} \; | \;  \text{there exists } v\in B_{t-1} \text{ such that }N(v) \cap (V(G)\setminus B_{t-1}) = \{u\} \},  \text{ and } \\
        B_t \;\; &= B_{t-1} \cup B_{(t)}.
    \end{align*}
    We refer to $t$ in $B_t$ as the {\em time} or {\em time step}.
The propagation time of a set $B \subseteq V(G)$ in $G$, denoted $\pt(G,B)$, is then the smallest value of $t$ such that $B_t = V(G)$, and the propagation time of $G$, denoted $\pt(G)$, is \[ \pt(G) = \min \{ \pt(G,B) \ | \ |B| = \Z(G)\}. \]    

\subsection{Transmission Zero Forcing}\label{subsec:tzf}

Let $G$ be a graph and fix two constants $\alpha, \beta \in [0,1]$.  We call $\alpha$ the {\em transmission proportion} and $\beta$ the {\em transmission threshold}.  Each vertex $v\in V(G)$ is assigned an initial weight, $w_0(v)$, in the following way. If $B \subseteq V(G)$, then $w_0(v) =1$ for all $v\in B$, and $w_0(u) =0$ for all $u \in V(G)\setminus B$. We denote the weight of a vertex after $t$ iterations of the transmission color change rule by $w_t(v)$. The weights of the vertices are then updated for $t\geq 1$ according to the transmission color change rule: A filled vertex $v$ can force an unfilled vertex $u$ if $u$ is the only unfilled neighbor of $v$, and we denote this force by $v\rightarrow u$. A vertex $v$ is filled at time $t$ if $w_t(v) \geq\beta$, and unfilled otherwise.  As with zero forcing, it is convenient to have notation to keep track of the evolution of the filled and unfilled vertices of the graph.

\begin{definition}
    Given a graph $G$ and a set $B \subseteq V(G)$, for $t \geq 1$ let:
    \begin{align*}
        B_0^{\ab} &= B_{(0)}^{\ab} = B \\
        B_{(t)}^{\ab} &= \{v \in V(G)\setminus B_{t-1}^{\ab} \ | \  w_t(v) \geq \beta  \}  \text{ and } \\
        B_t^{\ab} &= B_{t-1}^{\ab} \cup B_{(t)}^{\ab}.
    \end{align*}
    We refer to $t$ in the above as the {\em time} or {\em time step}.  The process terminates when $B_{(t+1)}^{\ab} = \emptyset$.  In this case, we define $w(v) = w_t(v)$ for each $v\in V(G)$.
\end{definition}

Then $B_t^{\ab}$ is the set of filled vertices at time $t$ and $V(G)\setminus B_t^{\ab}$
is the set of unfilled vertices at time $t$.  Let $\mathcal{V}_{t,u} = \{ v \in B_{t-1}^{\ab} \ | \ N(v) \cap (V(G)\setminus B_{t-1}^{\ab)})
=\{u\} \}$, or equivalently, $\mathcal{V}_{t,u} = \{ v \in B_{t-1}^{\ab} \ | \ v\rightarrow u \text{ is a valid force at time }t \}$. For each $u \in V(G)$, \begin{equation}\label{weightdef}
w_t(u) = \begin{cases}
    w_{t-1}(u) + \displaystyle\sum_{v\in \mathcal{V}_{t,u}} \alpha \cdot w_{t-1}(v), \quad u \in V(G)\setminus B_{t-1}^{\ab} \\ 
    w_{t-1}(u), \quad u \in B_{t-1}^{\ab}.
\end{cases}\end{equation}

We can now present the definition of the graph parameter studied in this paper.
\begin{definition}\label{def:Zab}
Given a graph $G$ and a set $B \subseteq V(G)$, we say that $B$ is a {\em $(\alpha,\beta)$-transmission zero forcing set} of $G$ if there exists $t\geq 0$ such that $B_t^{\ab} = V(G)$, i.e., every vertex of $G$ is eventually filled.  The {\em $(\alpha,\beta)$-transmission zero forcing number} of $G$, denoted $\Zab(G)$, is the minimum cardinality of an $(\alpha,\beta)$-transmission zero forcing set of $G$.
\end{definition}

We will often refer to the set (respectively, number) in Definition \ref{def:Zab} as transmission forcing set (respectively, number) for $(\alpha, \beta)$ and will omit the parameters when clear from context. 

\section{Transmission Zero Forcing Properties} \label{sec:transmission}

\Cref{def:Zab}  requires (when possible) multiple filled vertices to force a single unfilled vertex.  If this possibility were excluded (as in zero forcing), the exploration of $\Zab(G)$ would reduce to the study of $\pt(G,B)$ for various zero forcing sets of $G$, which has been well-studied (see, for instance, \cite{article}).  It is precisely this additivity of weights that provides transmission zero forcing with its richness.

\subsection{Basic Observations and Examples}

\begin{observation} \rm \label{basic} Let $G$ be a graph.
\begin{enumerate}
    \item \label{basic1} Any transmission zero forcing set must contain a zero forcing set.
    \item \label{basic2} $\Zab(G)\geq \Z(G)$.
    \item \label{basic3} If $\alpha =1$, then $\Zab(G) = \Z(G)$.
    \item \label{basic4} If $G$ is the disjoint union of $r$ connected graphs $\{G_i\}_{i=1}^r$, 
    then $\Zab(G) = \sum_{i=1}^r \Zab(G_i)$.
\end{enumerate}
\end{observation}

\begin{example} \label{ex:ygraph}
    The graph $G$ shown in Figure \ref{fig:ygraph} has $\Z(G) = 2$, and up to isomorphism, the three illustrations show the only three minimum zero forcing sets $B_1$, $B_2$, and $B_3$, hence $\pt(G)=2$. If $\alpha \geq 0.5$ and $\alpha^2 \geq \beta$, all three sets are transmission forcing sets for $(\alpha, \beta)$.  On the other hand, if $\alpha < \beta$ and $2\alpha^3 \geq \beta$ (for example, $\alpha = 0.8$ and $\beta = 0.9$), then the additivity in the first forcing step is necessary, and only $B_3$ is an $(\alpha, \beta)$-transmission forcing set.  
\end{example}

\begin{figure}[ht]
    \centering
\begin{tikzpicture}[scale=.95]
\vertex (1) at (-1, -1) [label= above: $\alpha^2$] {};
\bvertex (2) at (-1, 1) [label= above: $1$] {};    
\vertex (3) at (0,0) [label= above: $\alpha$] {};
\vertex (4) at (1,0) [label= above: $\alpha$]  {};
\bvertex (5) at (2, 0) [label= above: $1$] {};

\draw[<-, thick] (1) to (3);
\draw[->, thick] (2) to (3);
\draw[-, thick] (3) to (4);
\draw[<-, thick] (4) to (5);
\end{tikzpicture}
\hspace{.5in}  
\begin{tikzpicture}[scale=.95]
\vertex (1) at (-1, -1) [label= above: $\alpha^2$] {};
\bvertex (2) at (-1, 1) [label= above: $1$] {};    
\vertex (3) at (0,0) [label= above: $\alpha$] {};
\bvertex (4) at (1,0) [label= above: $1$]  {};
\vertex (5) at (2, 0) [label= above: $\alpha$] {};

\draw[<-, thick] (1) to (3);
\draw[->, thick] (2) to (3);
\draw[-, thick] (3) to (4);
\draw[->, thick] (4) to (5);
\end{tikzpicture}
 \hspace{0.5in}
\begin{tikzpicture}[scale=.95]
\bvertex (1) at (-1, -1) [label= above: $1$] {};
\bvertex (2) at (-1, 1) [label= above: $1$] {};    
\vertex (3) at (0,0) [label= above: $2\alpha$] {};
\vertex (4) at (1,0) [label= above: $2\alpha^2$]  {};
\vertex (5) at (2, 0) [label= above: $2\alpha^3$] {};

\draw[->, thick] (1) to (3);
\draw[->, thick] (2) to (3);
\draw[->, thick] (3) to (4);
\draw[->, thick] (4) to (5);
\end{tikzpicture}

    \caption{The graph $G$ in Example \ref{ex:ygraph}, with the initially filled sets $B_1$ (left), $B_2$ (middle), and $B_3$ (right) shown. }
    \label{fig:ygraph}
\end{figure}
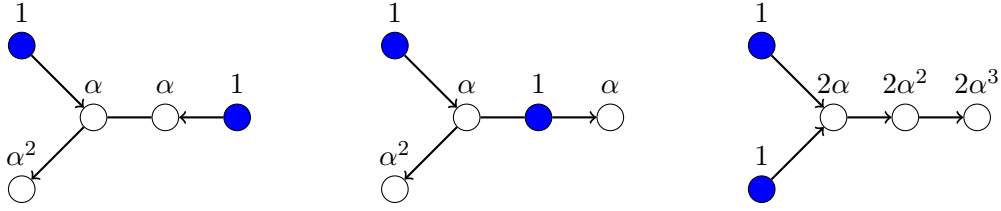

 \begin{observation} \label{obs:unique} \rm
Although it is known that there exists no vertex that must be included in all minimum zero forcing sets of a connected graph \cite{Barioli2010ZeroFP}, the set $B_3$ is the unique minimum  $(0.8, 0.9)$-transmission forcing set of the graph $G$ in \Cref{ex:ygraph}.  In \Cref{ex:conjpart2}, we will see one vertex that must be contained in any minimum transmission forcing set for the given $(\alpha, \beta)$.    

Furthermore, it is clear that if $B$ is a zero forcing set of a graph $G$, then any superset of $B$ is also a zero forcing set.  This property does not carry over to transmission zero forcing.  The set $B_3$ is a $(0.8, 0.9)$-transmission forcing set of the graph $G$ in \Cref{ex:ygraph}; however the union of $B_3$ and the degree-3 vertex is no longer a $(0.8, 0.9)$-transmission forcing set.  
\end{observation}

\begin{definition} \label{defn:proptime}
   Let $G$ be a graph and $B$ an $(\alpha,\beta)$-transmission zero forcing set of $G$.  The {\em $\ab$-propagation time} of $B$ in $G$, denoted $\pt^{\ab}(G,B)$, is the minimum value of $t$ such that $B_t^{\ab} = V(G)$. The {\em $\ab$-propagation time} of $G$, denoted $\pt^{\ab}(G)$, is  $\min \{ \pt^{\ab}(G,B) \ | \ |B| = \Zab(G)\}$.
\end{definition}

In Example~\ref{ex:ygraph}, we have $\pt^{(\frac{1}{2},\frac{1}{4})}(G, B_1)=2=\pt^{(\frac{1}{2},\frac{1}{4})}(G,B_2)$, and $\pt^{(\frac{1}{2},\frac{1}{4})}(G, B_3)=3$, hence $\pt^{(\frac{1}{2}, \frac{1}{4})}(G)=2$, while only $B_3$ transmission forces when $\alpha=.8$ and $\beta=.9$, hence $\pt^{(.8,.9)}(G)=3$. 

An important feature of zero forcing is a vertex needing to wait for all but one of its neighbors to be filled before it can perform a force.  In transmission zero forcing, this waiting may be exacerbated by a vertex also needing additional weight.  The next example illustrates this difference.

\begin{example} \label{ex:doubleygraph}
    The graph $G$ shown in Figure \ref{fig:doubleygraph} has $\Z(G) = 3$, and a minimum zero forcing set $B$ appears in the figure as the filled vertices.  If $\alpha^2 > \beta$ (for example, $\alpha = 0.9$ and $\beta = 0.8$), then the transmission forcing process proceeds in the same manner as the zero forcing process, with propagation time 3; vertex $v$ in the figure can perform a force as soon as vertex $u$ is filled.  However, if $\alpha < \beta$ and $2\alpha^2 \geq \beta$ and $\alpha^2 + 2\alpha^4 \geq \beta$ (for example, $\alpha = 0.7$ and $\beta = 0.8$), then vertex $v$ must wait not only for vertex $u$ to be filled but also for $u$ to force $v$, and the final filling is achieved in 4 time steps.
\end{example}

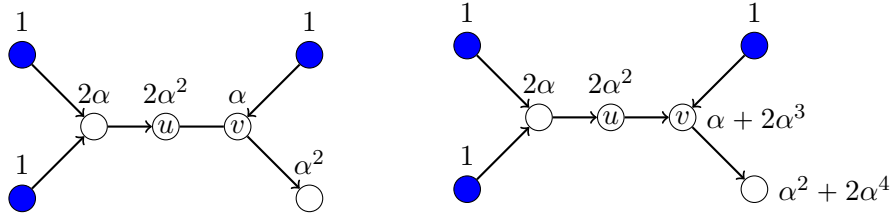
\begin{figure}[ht]
    \centering
\begin{tikzpicture}[scale=.95]
\bvertex (1) at (-1, -1) [label= above: $1$] {};
\bvertex (2) at (-1, 1) [label= above: $1$] {};    
\vertex (3) at (0,0) [label= above: $2\alpha$] {};
\vertex (4) at (1,0) [label= above: $2\alpha^2$]  {$u$};
\vertex (5) at (2, 0) [label= above: $\alpha$] {$v$};
\vertex (6) at (3, -1) [label= above: $\alpha^2$] {};
\bvertex (7) at (3, 1) [label= above: $1$] {}; 

\draw[->, thick] (1) to (3);
\draw[->, thick] (2) to (3);
\draw[->, thick] (3) to (4);
\draw[-, thick] (4) to (5);
\draw[->, thick] (7) to (5);
\draw[->, thick] (5) to (6);
\end{tikzpicture}
\hspace{.5in}   
\begin{tikzpicture}[scale=.95]
\bvertex (1) at (-1, -1) [label= above: $1$] {};
\bvertex (2) at (-1, 1) [label= above: $1$] {};    
\vertex (3) at (0,0) [label= above: $2\alpha$] {};
\vertex (4) at (1,0) [label= above: $2\alpha^2$]  {$u$};
\vertex (5) at (2, 0) [label= right: $\alpha+2\alpha^3$] {$v$};
\vertex (6) at (3, -1) [label= right: $\alpha^2+2\alpha^4$] {};
\bvertex (7) at (3, 1) [label= above: $1$] {}; 

\draw[->, thick] (1) to (3);
\draw[->, thick] (2) to (3);
\draw[->, thick] (3) to (4);
\draw[->, thick] (4) to (5);
\draw[->, thick] (7) to (5);
\draw[->, thick] (5) to (6);
\end{tikzpicture}

    \caption{The graph $G$ in Example \ref{ex:doubleygraph}, with the initially filled set $B$ shown and final filling achieved in 3 time steps (left) or 4 time steps (right).}
    \label{fig:doubleygraph}
\end{figure}

\begin{definition}
    Given a graph $G$ and set $B\subseteq V(G)$, the {\em set of forces}, denoted $\mathcal{F}(B)$, is the set of all forces $v\rightarrow u$ that produces the final filling according to the transmission color change rule.
\end{definition}

\begin{observation} \rm 
    Unlike in zero forcing, $\mathcal{F}(B)$ is unique for each transmission forcing set $B$.  
   
\end{observation}

\subsection{Size of Transmission Zero Forcing Sets} \label{sec:size} 

We investigate the extreme values of $\Zab(G)$, starting with the lower bound.  We begin by noting that results for zero forcing can sometimes be directly leveraged in the study of transmission zero forcing.

\begin{proposition}
    If $G$ is a graph on $n$ vertices and $\Zab(G)=1$, then $G=P_n$.
\end{proposition}
\begin{proof}
    This follows immediately from \Cref{basic} and the fact that $\Z(G)=1$ if and only if $G$ is a path \cite{AIM2008}.
\end{proof}

The converse of the previous proposition is false for transmission zero forcing and will be investigated in \Cref{thm:paths}.  Since $\delta(G) \leq \Z(G)$ for any graph $G$, \Cref{basic} immediately implies $\delta(G) \leq \Zab(G)$.  We turn our attention to showing that there are always values of $\alpha, \beta$ for which $\Zab(G)$ is maximal.

\begin{theorem} \label{z-is-n}
    For any graph $G$ on $n$ vertices, $\Z_{\alpha,\beta}(G)=n$ if any only if $\Delta(G)\alpha  <\beta$. 
\end{theorem}

\begin{proof} 
If $G$ is edgeless, then $\Zab(G) = \Z(G) = n$ and $\Delta(G)\alpha = 0 < \beta$, so assume that $G$ has an edge.

Suppose first that $\Delta(G)\alpha <\beta$ and that an  $(\alpha,\beta)$-transmission forcing set $B$ of $G$ satisfies $|B| \leq n-1$.  Then there exists a vertex $u \in B^{(\alpha,\beta)}_{(1)}$ that must be filled in the first time step.  Since $w_1(u) = m\cdot \alpha$, where $m$ is the number of vertices in $B$ that had $u$ as their unique unfilled neighbor, and $m\leq \Delta(G)$, we obtain a contradiction. 

Now suppose $\Delta(G)\alpha \geq \beta$. Let $u$ be vertex of degree $\Delta(G)\geq 1$. Let $B=V(G)\setminus \{u\}$. Then every vertex in $B$ has at most one unfilled neighbor, namely $u$.  So $u$ becomes filled in the first round, and thus $\Zab(G)\leq n-1$. 
\end{proof}

\Cref{z-is-n} highlights a fundamental difference between transmission zero forcing and standard zero forcing, where $\Z(G) \leq n-1$ if and only if $G$ has an edge.  The next result further improves \Cref{z-is-n}.

\begin{theorem} \label{thm:delta-1}
Let $G$ be a graph with maximum degree $\Delta = \Delta(G)$. If $\Delta \alpha\geq \beta>(\Delta -1)\alpha$, then $\Zab(G)=|V(G)\setminus S|$, where $S$ is the largest subset of vertices of $V(G)$ of degree $\Delta$, all of whose pairwise distances are at least 3. 
\end{theorem}

\begin{proof} Let $B$ be the initial set of filled vertices comprising an $(\alpha,\beta)$-transmission forcing set of $G$.  We claim that $V(G) = B^{(\alpha,\beta)}_1$, that is, $V(G) = B \cup B^{(\alpha,\beta)}_{(1)}$.  Since $\beta> (\Delta-1)\alpha$, only a vertex $x$ of degree $\Delta$ that is forced by all of its neighbors can be in $B^{(\alpha,\beta)}_{(1)}$, which implies $x$ cannot perform a force in a future time step.  So any vertex $b$ that performs a force in a time step $t\geq 2$ must belong to $B$ and not have performed a force in time step 1.  But this means that $b$ had at least 2 unfilled neighbors initially, all but one of which were filled during some time step by $\Delta$ vertices of $B$, implying that these unfilled neighbors of $b$ had degree $\Delta+1$.  This contradiction proves our claim.

Every vertex $x\in B^{(\alpha,\beta)}_{(1)}$ must have degree $\Delta$ and $N(x)\subseteq B$; moreover, for any $z\in N(x)$, we must have $N(z)\setminus \{x\} \subseteq B$ as well, since otherwise $z$ would have more than one unfilled neighbor at the start of time step 1.  It follows that for every $x,y\in B^{(\alpha,\beta)}_{(1)}$, $x$ and $y$ are not adjacent and $N(x)\cap N(y)=\emptyset$. Thus, the distance between any two vertices of $B^{(\alpha,\beta)}_{(1)}$ is at least 3.

Conversely, suppose that $S$ is a subset of vertices of $V(G)$ of degree $\Delta$, all of whose pairwise distances are at least 3. Let $B=V(G)\setminus S$ and $x\in S$.  Then $\deg(x) = \Delta$ and $N(x) \subseteq B$.  And if $y\in N(x)$, then $N(y)\cap S = \{x\}$, so $x\in B^{(\alpha,\beta)}_{(1)}$.  Hence, $S\subseteq B^{(\alpha,\beta)}_{(1)}$ and $V(G) = B^{(\alpha,\beta)}_1$.  So $B$ is a transmission forcing set for $(\alpha,\beta)$ if and only if $S = V(G)\setminus B$ is a subset of vertices of degree $\Delta$, all of whose pairwise distances are at least 3; the cardinality of $B$ is minimized when the cardinality of $S$ is maximized.
\end{proof}

\begin{corollary} \label{cor:delta-1}
    Suppose $G$ is a graph on $n$ vertices with $\mathrm{diam}(G) = 2$ and $\Delta(G) = \Delta$.  If $\Delta \alpha\geq \beta> (\Delta-1)\alpha$, then $\Zab(G)=n-1$.
\end{corollary}

\subsection{Forcing Forests}

In light of the different forcing processes illustrated in \Cref{ex:doubleygraph}, our next definition is crucial for transmission zero forcing. 
\begin{definition}
    Given a graph $G$ and $B \subseteq V(G)$, the {\em forcing subgraph} for $B$ in $G$ is $\Gamma=FF(G,B)$, where $V(\Gamma) = V(G)$ and $E(\Gamma) = \{ uv \ | \ u\rightarrow v \in \mathcal{F}(B) \}$. Since the forcing subgraph is determined by the set of forces, which are inherently directed, it is natural to view the forcing subgraph as a directed graph, which we denote by $\overrightarrow{FF}(G,B)$.
\end{definition}

Note that a vertex in $\overrightarrow{FF}(G,B)$ can have arbitrary in-degree, but it can have out-degree at most 1.  The next proposition describes the structure of forcing subgraphs. 

\begin{proposition}\label{forcing_forest}
    If $FF(G,B)$ is the forcing subgraph for the set $B$ in a graph $G$, then $FF(G,B)$ is a forest. 
\end{proposition}
\begin{proof}
    Suppose that $FF(G,B)$ contains a cycle $v_0 \rightarrow v_1 \rightarrow \dots \rightarrow v_r \rightarrow v_0$.  Then $v_0$ is forced after it has already performed a force, a contradiction.
\end{proof}

As a consequence of \Cref{forcing_forest}, we may now refer to the forcing subgraphs as forcing forests and their components as forcing trees.

\begin{proposition}\label{terminus} 
    Let $G$ be a graph, $B$ a transmission forcing set of $G$, and $FF(G,B)$ the forcing forest for $B$ in $G$.  For each forcing tree $T$ of $FF(G,B)$, let $S_T$ be the set of vertices that do not perform a force. Then $|S_T|=1$ for each forcing tree T.   
\end{proposition}
\begin{proof} 
    Suppose to the contrary that there are two vertices $z_1$ and $z_2$ in the component $T$ that do not perform a force. Since they each must have a neighbor in $T$, there exist vertices $u\neq v \in \overrightarrow{T}$ such that $u \rightarrow z_1, v\rightarrow z_2 \in \mathcal{F}(B)$.  There exists a path $P$ in $T$ between $u$ and $v$.  Since $u \rightarrow z_1$ (respectively, $v \rightarrow z_2$) is the only out-edge incident with $u$ (respectively, $v$), the directed path $\overrightarrow{P}$ in $\overrightarrow{T}$ must contain a vertex with two out-edges, a contradiction.
\end{proof}

In light of the previous proposition, we give a name to this unique element of $S_T$.

\begin{definition}
    Let $G$ be a graph, $B$ a transmission forcing set of $G$, and $FF(G,B)$ the forcing forest for $B$ in $G$.  For each component $T$ in $FF(G,B)$, the unique element that does not perform a force is called the {\em terminus} of $T$.
\end{definition}

Taking the terminus $z$ as the root of a directed forcing tree $\ora{T} \in \ora{FF}(G,B)$, the direction of each edge of $\ora{T}$ is toward $z$; that is, $\ora{T}$ is an in-tree \cite[p. 89]{west1996intro}. Also, the distance from any leaf in $\ora{T}$ to $z$ is the same whether $T$ is directed or undirected. 

\begin{definition} \label{ffarrow}
    Let $G$ be a graph and $B \subseteq V(G)$ a set of initially filled vertices.  The set of {\em initial vertices that contribute weight} to $v\in V(G)$, denoted $R_v \subseteq B$, is the set of vertices $b \in B$ for which there exists a directed path from  $b$ to $v$ in  $\ora{FF}(G,B)$. 
\end{definition}

Note that for each $v\in V(G)\setminus B$, $R_v$ is nonempty whenever $B$ is a transmission forcing set.

\begin{observation} \label{prop:leaves}
    Suppose $\overrightarrow{T_i}$ is a forcing tree in $\ora{FF}(G,B)$ with terminus $z_i$ and let $L_i = V(T_i)\cap B$.  Then for every $v\in V(T_i)$, $R_v \subseteq L_i$; in particular, $R_{z_i} = L_i$.
\end{observation}

While the forcing subgraphs are indeed composed of trees, these trees need not be induced subgraphs of the graph, as seen in the next example. In contrast, the analogous structure in zero forcing are zero forcing chain sets, which form induced (vertex) path covers of $G$, see \cite{AIM2008}.

\begin{example}\rm \label{ex:ff}
    Let $G= K_n$ and set $\alpha = \beta =1$.  Then $\Zab(G) = \Z(G) = n-1$.  If $B$ is a transmission forcing set of $G$ with cardinality $n-1$, then $FF(G,B) = K_{1,n-1}$. 
\end{example}

Even if the graph is a tree, its forcing forest may have multiple components.  In \Cref{fig:ygraph}, $FF(G,B_1)$ and $FF(G,B_2)$ are each the union of two trees while $FF(G,B_3)$ is a single tree.  In \Cref{fig:doubleygraph}, $FF(G,B)$ for the graph on the left is the union of two trees while the forcing forest on the right is a single tree.

The final weight of each vertex depends both on propagation distances and additivity of weights.  We end this subsection with some technical results showcasing these issues.

\begin{observation}\rm \label{ob:u-weight}
Let $G$ be a graph, $B \subseteq V(G)$, and $\overrightarrow{F}= \overrightarrow{FF}(G,B)$.
\begin{enumerate}
    \item For each $v\in V(G)$, $w(v) = \displaystyle \sum_{b \in R_v} \alpha^{\scriptsize \ora{\dist}_{\scriptsize \ora{F}}(b,v)}$. 
    \item  For each $v \not\in B$, $w(v)=0$ or $|R_v| \cdot \alpha \geq w(v) \geq |R_v| \cdot \alpha^k $, \;$k = \displaystyle \max_{b\in R_v}\;\ora{\dist}_{\scriptsize \ora{F}}(b,v)$.
\end{enumerate}
     
\end{observation}  

\begin{corollary} \label{cor-t} 
    If $\alpha^{\pt(G)} \geq \beta$, then $\Zab(G) = \Z(G)$.
\end{corollary}
\begin{proof}
    Let $B$ be a minimum zero forcing set of $G$ such that $\pt(G,B) =\pt(G)$.  Allow the $(\alpha, \beta)$-transmission forcing process to proceed from the initially filled set $B$.  If $v\in B$, $w(v) = 1$, and for $v\in V(G)\setminus B$, $w(v) \geq \alpha^{\pt(G)}$.  
    So $\Zab(G) \leq \Z(G)$. \Cref{basic} establishes the reverse inequality.
\end{proof}

Note that the converse of \Cref{cor-t} need not hold. For example, since $K_3$ is complete, $\Z(K_3)=2$ and $\pt(K_{3})=1$. Let $\alpha = \frac{1}{2}$ and $\beta = \frac{2}{3}$, then  $\Zab(K_3)=\Z(K_3)$, with the same propagation time, but $(\frac{1}{2})^1<\frac{2}{3}$.

\begin{definition}
    Let $B\subseteq V(G)$ be a transmission forcing set of $G$, let $FF(G,B)$ be the forcing forest for $B$ in $G$, and let $\{T_i\}_{i=1}^q$ be the set of forcing trees in $FF(G,B)$.  Denote by $z_i$ the terminus of $T_i$ for $1\leq i\leq q$.  

      \begin{enumerate}
        \item Let $m(T_i)$ denote the minimum forcing distance of $T_i$, i.e. $m(T_i)=\min\{ \dist_{T_i}(b,z_i) \ | \ b \in B \}$.  
        \item Let $m^{\ell}(G,B)$ denote the largest minimum forcing distance of $B$ in $G$, i.e. $m^{\ell}(G,B)= \max\{ m(T_i) \ | \ 1\leq i\leq q\}$. 
    \end{enumerate}
\end{definition}

\begin{proposition}
Let $G$ be a graph and $B \subseteq V(G)$ be an $(\alpha, \beta)$-transmission forcing set. 
For each $v\in V(G)$, $w(v) \geq \alpha^{m^{\ell}(G,B)}$.
\end{proposition}

\begin{proof}
    Let $\{T_i\}_{i=1}^q$ be the set of forcing trees in $FF(G,B)$.  Denote by $z_i$ the terminus of $T_i$ for $1\leq i\leq q$.  Since $m(T_i)\leq m^{\ell}(G,B)$ and $0<\alpha<1$, we have $\alpha^{m(T_i)} \geq \alpha^{m^{\ell}(G,B)}$, so we need only prove that for each $T_i\in FF(G,B)$ and $v\in T_i$, $w(v) \geq \alpha^{m(T_i)}$.  Suppose to the contrary that $w(v) < \alpha^{m(T_i)}$ for some $v$.  Since $B$ is a transmission forcing set, $\beta\leq w(v)$, and therefore $\beta <\alpha^{m(T_i)}$.  Since $w(z_i) \geq \alpha^{m(T_i)}$, this implies $z_i$ is filled at timestep $m(T_i)$, and hence every vertex, including $v$, is filled at timestep $m(T_i)$.  This implies $w(v) \geq \alpha^{m(T_i)}$, a contradiction.
\end{proof}

\begin{theorem} \label{thm:not-B} 
Let $G$ be a graph with $n$ vertices and $B\subseteq V(G)$. Let $0<\alpha,\beta<1$ and $j$ be such that $\alpha^{j}< \beta \leq \alpha^{j-1}$ with $j>0$.
If $B$ is an $(\alpha,\beta)$-transmission forcing set of $G$, then $n\leq|B|\cdot j$ or $\beta\leq|B| \cdot \alpha^j$. 
\end{theorem} 
\begin{proof}
Assume $B$ is a transmission forcing set for $(\alpha,\beta)$.  If $B = V(G)$, then $n = |B|\leq |B|\cdot j$ since $j\geq 1$.  Suppose $B \subsetneq V(G)$.  By \Cref{ob:u-weight}, if $R_u = \{y_1, y_2, \ldots, y_k\}$ are the vertices in $B$ that contribute to the weight of a vertex $u$, then $w(u) =\sum_{i=1}^k \alpha^{j_i}$, where $j_i$ is the length of the directed path from $y_i$ to $u$ in $\overrightarrow{FF}(G,B)$.  If there exists $u\in V(G)\setminus B$ such that the corresponding lengths satisfy $j_i\geq j$ for $1\leq i\leq k\leq |B|$, then $\beta \leq w(u)\leq k\alpha^j\leq |B|\cdot \alpha^j$.  Otherwise, for each $u\in V(G)\setminus B$, there exists a directed path from a vertex in $B$ to $u$ of length at most $j-1$.  These paths form an edge cover of $G$.  Since a vertex can force at most one other vertex, each vertex of $B$ is the first vertex of at most one such directed path.  Each path contains at most $j$ vertices, so $n\leq |B|\cdot j$.  
\end{proof}

Although we have seen that the converse of \Cref{cor-t} is false, \Cref{thm:not-B} suggests the following:
 If $\Zab(G) = \Z(G)$ and $\pt^{\ab}(G) =\pt(G)$, is it true that $\Z(G)\alpha^{\pt(G)} \geq \beta$? The answer to this is negative, as shown in the next example, which also highlights the difference between $\pt^{\ab}(G,B)$ and $m^{\ell}(G,B)$ for a given set $B$ of initially filled vertices.

\begin{example} \label{exa:fig-Zpt}
Let $G$ be the graph with initially filled vertices $B$, as shown in Figure~\ref{fig:exa-Zpt}. This example shows that when $\Zab(G)=\Z(G)$, the value $\Z(G) \alpha^{\pt(G)}$ can be less than $\beta$. Let $\alpha=.7$ and $\beta=.3$. Then $B$ is a transmission forcing because the smallest weight is $\alpha^{m^{\ell}(G,B)} = \alpha^3 = .7^3>.3$. However, $\Z(G)=2$ and $\pt(G)=\pt^{\ab}(G)=6$ and $2(.7)^6<.3$.
\end{example}

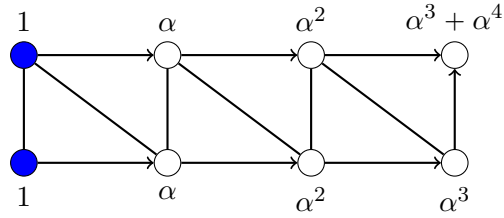
\begin{figure}[ht]\centering 
\begin{tikzpicture}[scale=.95]

\bvertex (1) at (0,0) [label= below: $1$] {};
\bvertex (2) at (0,1.5) [label= above: $1$] {};
\vertex (3) at (2,0) [label= below: $\alpha$] {};
\vertex (4) at (2,1.5) [label= above: $\alpha$] {};
\vertex (5) at (4,0) [label= below: $\alpha^2$] {};
\vertex (6) at (4,1.5) [label= above: $\alpha^2$] {};
\vertex (7) at (6,0) [label= below: $\alpha^3$] {};
\vertex (8) at (6,1.5) [label= above: $\alpha^3+\alpha^4$] {};

\draw[->,thick] (1) to (3);
\draw[->,thick] (2) to (4);
\draw[->,thick] (3) to (5);
\draw[->,thick] (4) to (6);
\draw[->,thick] (5) to (7);
\draw[->,thick] (6) to (8);
\draw[->,thick] (7) to (8);

\draw[thick] (1) to (2);
\draw[thick] (2) to (3);
\draw[thick] (3) to (4);
\draw[thick] (4) to (5);
\draw[thick] (5) to (6);
\draw[thick] (6) to (7);
\end{tikzpicture}
\caption{The graph $G$ in Example~\ref{exa:fig-Zpt}.}

\label{fig:exa-Zpt}
\end{figure}

The connection between propagation time for standard zero forcing, propagation time for transmission zero forcing, and the minimum weight of a filled vertex is subtle. The next two examples show us that the propagation time also heavily depends on the choice of $\alpha$ and $\beta$.

\begin{example} \label{ex:conj}
    Let $G$ be the graph in \Cref{fig-tree-wait}.  Let $\alpha = 0.45$ and $\beta = 0.21$.  Note that $\Z(G) = \Zab(G) = 4$ and the filled vertices in the figure are a minimum zero forcing and minimum $(\alpha,\beta)$-transmission forcing set with propagation time (for both processes) equal to 3.  Also, the hypothesis of \Cref{thm:not-B} is satisfied with $j=2$; note that $9 > 4\cdot 2$ but $0.21 \leq 4\cdot 0.45^2$.  Moreover, $4\cdot 0.45^3 \geq 0.21$.
\end{example}

\begin{figure}[ht]\centering 
\begin{tikzpicture}[scale=.8]
\vertex (1) at (-.15,0) [label= above: $2\alpha$] {};
\bvertex (8) at (1.2, 0) [label= above: $1$] {};
\vertex (2) at (3,0) [label= above: $\alpha+\alpha^2$] {};
\vertex (3) at (5.3, 0) [label= above: $\alpha^2+\alpha^3$] {};
\vertex (4) at (3, -1) [label= right: $\alpha$] {};
\bvertex (5) at (3, -2) [label= below: $1$] {};
\vertex (6) at (-1, -1) [label= below: $2\alpha^2$] {};
\bvertex (7) at (-1, 1) [label= above: $1$] {};

\bvertex (0) at (-2, 0) [label=below: $1$]  {};

\draw[->,thick] (0) to (1);
\draw[->,thick] (7) to (1);
\draw[<-,thick] (6) to (1);
\draw[-, thick] (1) to (8);
\draw[->, thick] (8) to (2);
\draw[->, thick] (2) to (3);
\draw[->, thick] (5) to (4);
\draw[->, thick] (4) to (2);

\end{tikzpicture}
\caption{The graph $G$ in \Cref{ex:conj}.}

\label{fig-tree-wait}
\end{figure}

\begin{example} \label{ex:conjpart2}
Consider the graph $G$ in \Cref{fig:alpha79beta99}, and let $\alpha = 0.79$ and $\beta = 0.99$.  The set $B_1$ of initially filled vertices, shown in the top graph of the figure, is a minimum zero forcing set such that $\pt(G,B_1) = \pt(G) = 3$.  However, $B_1$ is not a transmission forcing set for the given $(\alpha,\beta)$.  The set $B_2$ of initially filled vertices, shown in the bottom graph of the figure, is a minimum $(\alpha,\beta)$-transmission forcing set, with propagation time 7.  Note that $\Z(G)\alpha^t \geq \beta$ is satisfied with $t=3$ but not with $t=7$.  

As pointed out in \Cref{obs:unique}, in contrast with zero forcing, any minimum transmission forcing set for $(.79,.99)$ must contain the central leaf vertex.  Indeed, if this vertex were not in $B_2$, then its degree-3 neighbor, $v$, would have to force it, which would require $v$ to be forced from both sides, which is impossible since $2\alpha^3 < \beta$.
\end{example}

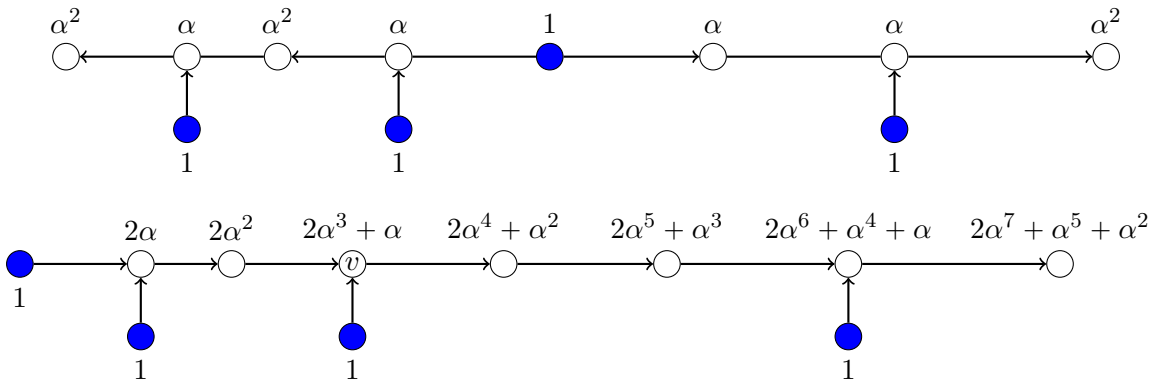
\begin{figure}[htb]\centering 

\begin{tikzpicture}[scale=.8]
\vertex (1) at (-3,0) [label= above: $\alpha$] {};
\vertex (z) at (-1.5,0) [label= above: $\alpha^2$] {};
\vertex (2) at (.5, 0) [label= above: $\alpha$] {};
\bvertex (3) at (3,0) [label= above: $1$] {};
\vertex (4) at (5.7,0) [label= above: $\alpha$] {};
\vertex (5) at (8.7, 0) [label= above: $\alpha$] {};
\bvertex (u) at (8.7, -1.2) [label= below: $1$] {};
\bvertex (k) at (.5, -1.2) [label= below: $1$] {};
\vertex (y) at (12.2, 0) [label= above: $\alpha^2$] {};
\bvertex (i) at (-3, -1.2) [label= below: $1$] { };
\vertex (h) at (-5, 0) [label= above: $\alpha^2$] {};

\draw[<-,thick] (h) to (1);
\draw[->,thick] (i) to (1);
\draw[-, thick] (1) to (z);
\draw[<-, thick] (z) to (2);
\draw[-, thick] (2) to (3);
\draw[->, thick] (3) to (4);
\draw[->, thick] (k) to (2);
\draw[->, thick] (u) to (5);
\draw[-, thick] (4) to (5);
\draw[->, thick] (5) to (y);

\end{tikzpicture}

\vspace{0.1in}

\begin{tikzpicture}[scale=.8]
\vertex (1) at (-3,0) [label= above: $2\alpha$] {};
\vertex (z) at (-1.5,0) [label= above: $2\alpha^2$] {};
\vertex (2) at (.5, 0) [label= above: $2\alpha^3+\alpha$] {$v$};
\vertex (3) at (3,0) [label= above: $2\alpha^4+\alpha^2$] {};
\vertex (4) at (5.7,0) [label= above: $2\alpha^5+\alpha^3$] {};
\vertex (5) at (8.7, 0) [label= above: $2\alpha^6+\alpha^4+\alpha$] {};
\bvertex (u) at (8.7, -1.2) [label= below: $1$] {};
\bvertex (k) at (.5, -1.2) [label= below: $1$] {};
\vertex (y) at (12.2, 0) [label= above: $2\alpha^7+\alpha^5+\alpha^2$] {};
\bvertex (i) at (-3, -1.2) [label= below: $1$] {};
\bvertex (h) at (-5, 0) [label= below: $1$] {};


\draw[->,thick] (h) to (1);
\draw[->,thick] (i) to (1);
\draw[->, thick] (1) to (z);
\draw[->, thick] (z) to (2);
\draw[->, thick] (2) to (3);
\draw[->, thick] (3) to (4);
\draw[->, thick] (k) to (2);
\draw[->, thick] (u) to (5);
\draw[->, thick] (4) to (5);
\draw[->, thick] (5) to (y);

\end{tikzpicture}
\caption{The graph $G$ in \Cref{ex:conjpart2} with the initially filled sets $B_1$ (top) and $B_2$ (bottom).}

\label{fig:alpha79beta99}
\end{figure}

\section{Graph Families}\label{sec:families}
In this section we consider the transmission zero forcing number of several graph families, drawing on the results for $\Z(G)$ and \Cref{basic}.

\subsection{Complete Graphs and Complete Bipartite Graphs}
If $G$ is a complete graph, then $\Z(G) = n-1$.  Thus the characterization of $\Zab(G)$ is a direct consequence of \Cref{z-is-n}.
\begin{proposition}
    For $n\geq 2$, \[ \Z_{\alpha, \beta}(K_n)=\begin{cases} 
      n-1 & \beta \leq (n-1)\alpha\\
      n & (n-1)\alpha<\beta. \\    
   \end{cases}
\]
\end{proposition}

\begin{proof}
For $n\geq 2$, $K_n$ contains an edge and $\Z(K_n) =n-1$. Thus, by \Cref{z-is-n}, $\Zab(K_n)=n$ if and only if $\Delta(G)\alpha <\beta$, and otherwise, $\Zab(G)=n-1$.
\end{proof}

\begin{proposition} \label{prop:star}
    For $n\geq 3$, \[ \Z_{\alpha, \beta}(K_{1,n-1})=\begin{cases} 
      n-2 & \beta \leq (n-2)\alpha^2\\
      n-1 & (n-2)\alpha^2 < \beta \leq (n-1)\alpha\\
      n & (n-1)\alpha<\beta. \\  
   \end{cases}
\]
\end{proposition}

\begin{proof}
For $n\geq 3$, $K_{1,n-1}$ contains at least 2 edges and $\Delta(K_{1,n-1})=n-1$.  It is well-known that $\Z(K_{1,n-1})=n-2$ and any minimum zero forcing set consists of (any) $n-2$ vertices of degree 1.  Therefore, if $\Zab(K_{1,n-1})=n-2$, the transmission forcing set must be a zero forcing set. After coloring is applied to this set, the weights of the initially unfilled vertices are $(n-2)\alpha$ and $(n-2)\alpha^2$.  Thus $\Zab(K_{1,n-1})=n-2$ if and only if $(n-2)\alpha^2 \geq \beta$. By \Cref{z-is-n}, $\Zab(K_{1,n-1})=n$ if and only if $(n-1)\alpha <\beta$, which also determines the remaining case.  
\end{proof}

\begin{proposition} 
For $n\geq m\geq 2$,
\[ \Z_{\alpha, \beta}(K_{m,n})=\begin{cases} 
      m+n-2 &  \beta\leq (m-1)\alpha \\ m+n-2& (m-1)\alpha<\beta\leq (n-1)\alpha \mbox{ and } \beta \leq (m-1)\alpha+ (n-1)\alpha^2\\
      m+n-1 & (n-1)\alpha<\beta\leq n\alpha \mbox{ or } (m-1)\alpha +(n-1)\alpha^2< \beta \leq n\alpha\\ 
      m+n & n\alpha < \beta. \\  
   \end{cases}
\]
\end{proposition}

\begin{proof}
For $n\geq m\geq 2$,  $\Z(K_{m,n})=m+n-2$ \cite{AIM2008}.  By \Cref{z-is-n}, $\Zab(K_{m,n})=m+n$ if and only if $n\alpha<\beta$. Note that any minimum zero forcing set $B$ consists of $m-1$ vertices of degree $n$ and $n-1$ vertices of degree $m$. Let $x,y \in V(G)\setminus B$ such that $\deg(x)=n$ and $\deg(y)=m$. After applying one coloring step to this set, the weights of the initially unfilled vertices are $w_1(x)=(n-1)\alpha$ and $w_1(y)=(m-1)\alpha$. Since $(m-1)\alpha \leq (n-1)\alpha$, we have that if $\beta \leq (m-1)\alpha$, then $\Zab(K_{m,n})=m+n-2$. 

If $(m-1)\alpha< \beta \leq (n-1)\alpha$, then $x\rightarrow y$ is valid at the second time-step, so
$w_2(y)=(m-1)\alpha+(n-1)\alpha^2$, and $y$ is filled if $(m-1)\alpha +(n-1)\alpha^2\geq \beta$.  These are the only situations where $\Zab(K_{m,n})=m+n-2$ since if $(n-1)\alpha<\beta$, neither $x$ nor $y$ is filled, and the transmission forcing process cannot continue. 
The remaining case is then determined by the negations of the conditions for $\Zab(K_{m,n})$ to be $m+n-2$ or $m+n$.
\end{proof} 

\begin{example}
    Let $G= K_{3,2}$.  Then $\Z_{(\frac{1}{2},\frac{1}{2})}(G) = 3$, $\Z_{(\frac{3}{10},\frac{1}{2})}(G)=4$, and $\Z_{(\frac{1}{10},\frac{1}{2})}(G) = 5$.
\end{example}
\subsection{Paths and Cycles}\label{paths}
If $G$ is a path or cycle, then $\Delta(G)=2$.  Thus, the forcing forest for any transmission forcing set of $G$ is composed of paths, while each component of the directed forcing forest can be the union of two directed paths with the same terminus. This simplifies the computation of the transmission zero forcing number for these graphs.

For any integers $k$ and $j$, let $k_{(j)}=k\;(\bmod{\;j})$, with $0\leq k_{(j)}\leq j-1$. If $B \subseteq V(G)$, use $\min w_{\ab}(B,G)$ (or simply $\min w(B,G)$) to denote the minimum weight among all vertices of $G$ after the forcing process initiated by $B$ has terminated.  Clearly $B$ is a $\ab$-transmission forcing set if and only if $\beta \leq \min w(B,G)$.

\begin{theorem} \label{thm:paths} Let $0< \alpha <1$ and let $j$ be such that $\alpha^{j}< \beta \leq \alpha^{j-1}.$ If $2\alpha^{j}<\beta$,  then $\Zab(P_n)=\lceil \frac{n}{j}\rceil $, and if $\beta \leq 2\alpha^j$, then 
$\Zab(P_n)= \lceil \frac{2n}{2j+1}\rceil $. Furthermore, for $n>j$, there is a minimum transmission forcing set that includes both endpoints of $P_n$.
\end{theorem} 

\begin{proof} Let $V(P_n)=\{v_0, v_1, \ldots, v_{n-1}\}$ and $E(P_n) = \{ v_iv_{i+1} \ | \ i \in \{ 0,1, \dots, n-2\} \}$. We will first show that there exists a transmission forcing set of the appropriate size, and then show that the size is the minimum.

For the first case, suppose $2\alpha^{j} < \beta$. Let $S=\{v_i\;|\;i\equiv 0 \;  (\bmod{\;j})\}$. Then $|S|=\lceil \frac{n}{j}\rceil$ and when using $S$ as the initial set of filled vertices, each vertex receives a weight of at least $\alpha^{j-1}$ in the forcing process. It follows that $S$ is a transmission forcing set and $\Zab(P_n) \leq \lceil \frac{n}{j}\rceil$. Observe that for the set $S$ constructed, the largest index of an initially filled vertex is $\lfloor \frac{n-1}{j} \rfloor \cdot j$.  If $S^{\prime} = S \setminus \{ v_{\lfloor \frac{n-1}{j} \rfloor \cdot j} \} \cup \{ v_{n-1}\}$, then $\min w(S^{\prime},P_n) \geq \min w(S,P_n)$, so $S'$ is a transmission forcing set that includes both endpoints of $P_n$ if $n>j$, and $|S'|=|S|$. 

For the second case, assume $2\alpha^{j}\geq \beta$. If $n_{(2j+1)} > j$, let  $S=\{v_i\;|\;i\equiv 0,2j \;  (\bmod{\;2j+1}), \ i \leq \lfloor \frac{n}{2j+1} \rfloor\cdot (2j+1) \}\cup \{v_{n-1}\}$.  If $n_{(2j+1)}\leq j$, let $S=\{v_i\;|\;i\equiv 0,2j \;  (\bmod{\;2j+1}), \ i < \lfloor \frac{n}{2j+1} \rfloor\cdot (2j+1) \}\cup \{v_{n-1}\}$.
It is straightforward to check that $|S|=\lceil \frac{2n}{2j+1}\rceil$. The graph $FF(P_n,S)$ is $\lfloor \frac{n}{2j+1} \rfloor P_{2j+1} \sqcup P_{n_{(2j+1)}}$.   In either case, $\min w(S,P_n) = \min \{\alpha^{j-1}, 2\alpha^j\} \geq \beta$. Thus, $S$ is a  transmission forcing set and $\Zab(P_n)\leq \lceil \frac{2n}{2j+1}\rceil$. 

 For the reverse inequality, let $B$ be the filled vertices at time step 0 in an $(\alpha,\beta)$-forcing set. Since $FF(P_n,B)$ is a subgraph of $P_n$, every component of $FF(P_n,B)$ is a path. By Proposition~\ref{prop:leaves}, for any $u\not \in B$, we have $|R_u|\leq 2$. By Observation~\ref{ob:u-weight}(1), for $u\not\in B$, we have $w(u)=\alpha^i$ for some $i$, $1\leq i\leq j-1$ with $|R_u|=1$, or $w(u) = \alpha^{i_1}+\alpha^{i_2}$ where $\alpha^{i_1}+\alpha^{i_2}\geq \beta$ and $|R_u|=2$. First, assume that $2\alpha^j < \beta$. If both $i_1,i_2\geq j$, then $\alpha^{i_1}+\alpha^{i_2}\leq 2\alpha^j < \beta$, hence at least one of $i_1, i_2\leq j-1$. Thus, for every vertex $u\not\in B$, there is a vertex $v_u\in B$ for which there is a path from $v_u$ to $u$ in $\overrightarrow{FF}(P_n,B)$ of length less than or equal to $j-1$, with $j$ or fewer vertices. Hence the total number of vertices in $P_n$ is less than or equal to $j\cdot |B|$. Thus $|B|\geq \lceil \frac{n}{j}\rceil$. 
Second, assume that $2\alpha^j \geq\beta$. In addition to the values described above, $w(u)$ could equal $2\alpha^j$, with $|R_u|=2$ and two directed paths of length $j$ ending at $u$ in $\overrightarrow{FF}(P_n,B)$. Hence the total number of vertices in $P_n$ is less than or equal to $j\cdot |B| + \frac{|B|}{2}$, and we have $|B|\geq \lceil \frac{n}{j+\frac{1}{2}}\rceil=\lceil\frac{2n}{2j+1} \rceil$. 
\end{proof}

\begin{theorem} \label{thm:cycles} Let $0< \alpha <1$ and let $j$ be such that $\alpha^{j}< \beta \leq \alpha^{j-1}.$
If $n>j$ and $2\alpha^{j} < \beta$,  then $\Zab(C_n)=\lceil \frac{n}{j}\rceil $, and if $n>j$ and $\beta\leq 2\alpha^{j}$, then 
$\Zab(C_n)= \lceil \frac{2n}{2j+1}\rceil $.

If $n\leq j$, then $\Zab(C_n)=2$.
\end{theorem} 
\begin{proof}

Suppose first that $n>j$.  Recall from \Cref{basic} that any transmission forcing set of a graph must contain a zero forcing set.  It is straightforward to see that for the cycle $C_n$ on $n\geq 3$ vertices, a zero forcing set must contain a pair of adjacent vertices.  Thus every minimum transmission forcing set $B$ of $C_n$ contains a pair of adjacent vertices that we may take, without loss of generality, to be $\{v_0,v_{n-1}\}$. Since these vertices cannot contribute to the weight of one another, we have $FF(C_n,B) \cong FF(C_n-v_0v_{n-1},B)$.  Then $B$ is a transmission forcing set of $C_n-v_0v_{n-1} \cong P_n$ and $\Zab(C_n) \geq \Zab(P_n)$. 

Now suppose $B$ is a minimum transmission forcing set of $P_n$.  \Cref{thm:paths} allows us to replace $B$ with a transmission forcing set of the same size containing both endpoints $\{v_0,v_{n-1}\}$ of $P_n$. Adding the edge $v_0v_{n-1}$ does not change the forcing forest, and thus $B$ is a transmission forcing set of $(V(P_n), E(P_n)\cup \{v_0v_{n-1}\}) \cong C_n$.  Therefore $\Zab(P_n) \geq \Zab(C_n)$.

Now suppose $n\leq j$.  Since $\pt(C_n) = \lceil\frac{n-2}{2}\rceil < n$ \cite{chilakamarri}, $\Zab(C_n)=\Z(C_n)=2 $ follows from \Cref{cor-t}. 
\end{proof}

\section{Effect of Graph Minor Operations}\label{sec:vert_edge_del}

It is useful to understand the behavior of graph parameters under graph minor operations, i.e.,  vertex deletion, edge deletion, and edge contraction, and in this section we study the effect of these operations on the transmission zero forcing number. We begin by summarizing the analogous results for zero forcing.

\begin{theorem}[\cite{edholm2012},\cite{owens}] 
    Let $G$ be a graph, $e$ an edge of $G$, and $v$ a vertex of $G$.  
    \begin{enumerate}
        \item $-1 \leq \Z(G) -\Z(G-v) \leq 1$,
        \item $-1 \leq \Z(G) -\Z(G-e) \leq 1$, 
        \item $-1 \leq \Z(G) - \Z(G\slash e)\leq 1$,
        
    \end{enumerate}
    and all bounds are sharp.
\end{theorem}

Thus, $\Z(G)$ is not a minor monotone graph parameter. A useful variant of $\Z(G)$ exhibiting minor monotonicity has been studied in, for example, \cite[Section 2.F]{parameters2012}. Graph minor operations have a greater impact on the transmission zero forcing number because they can interrupt the transfer of high weights among the vertices. 
To establish our results, we will make use of variants of graphs often referred to as broom graphs, which we now define.

\begin{definition} \label{defn:star-path} \rm Let $G(\ell,k)$  be the graph constructed by identifying a degree-1 vertex of the path $P_{k+1}$ with the central vertex of the star $K_{1,\ell}$.  Call the identified vertex $v$ and its neighbor on the path $x$. The graph is illustrated in Figure~\ref{exa:star-path}. 
\end{definition}

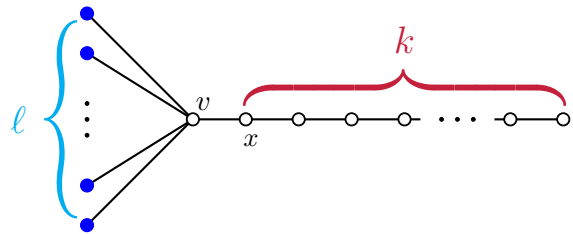
\begin{figure}[ht]
\centering
\begin{tikzpicture}[scale=.7]
\draw[thick] (0,1.75)--(2,3)--(6.3,3);
\draw[thick] (0,4.25)--(2,3);
\draw[thick] (0,1)--(2,3)--(0,5);
\draw[thick] (7.7,3)--(9,3);

\filldraw[blue]

(0,1.75) circle [radius=3.5pt]
(0,1) circle [radius=3.5pt]
(0,5) circle [radius=3.5pt]
(0,4.25) circle [radius=3.5pt]
;
\filldraw[black]
(2,3) circle [radius=3.5pt]
(3,3) circle [radius=3.5pt]
(4,3) circle [radius=3.5pt]
(5,3) circle [radius=3.5pt]
(6,3) circle [radius=3.5pt]

(8,3) circle [radius=3.5pt]
(9,3) circle [radius=3.5pt]

(0,3.3) circle [radius=.7pt]
(0,3) circle [radius=.7pt]
(0,2.7) circle [radius=.7pt]
(6.7,3) circle [radius=.7pt]
(7,3) circle [radius=.7pt]
(7.3,3) circle [radius=.7pt]
;

\filldraw[white]

(2,3) circle [radius=2.5pt]
(3,3) circle [radius=2.5pt]
(4,3) circle [radius=2.5pt]
(5,3) circle [radius=2.5pt]
(6,3) circle [radius=2.5pt]

(8,3) circle [radius=2.5pt]
(9,3) circle [radius=2.5pt]
;
\filldraw[black]
(0,3.3) circle [radius=1pt]
(0,3) circle [radius=1pt]
(0,2.7) circle [radius=1pt]
(6.7,3) circle [radius=1pt]
(7,3) circle [radius=1pt]
(7.3,3) circle [radius=1pt]
;

\node(0) at (-.5,3) {\Huge \color{cyan} $\Biggl\{$};
\node(1) at (-1.3,3) {\Large \color{cyan} $\ell$};
\node(2) at (6,3.6) {\Huge \color{cardinal} $\overbrace{\;\;\;\;\;\;\;\;\;\;\;\;\;\;\;\;\;\;}$};
\node(3) at (6,4.5) {\Large \color{cardinal} $k$};

\node(4) at (3.1,2.6) { \small \color{black} $x$};

\node(5) at (2.2,3.3) { $v$};
\node(6) at (9,2.6) {};

\end{tikzpicture}

\caption{Illustration of the graph $G(\ell,k)$ and a minimum zero forcing set of size $\ell$ that consists of all the leaves adjacent to $v$ indicated by filled vertices. }
\label{exa:star-path}
\end{figure}

\begin{observation} \label{obs:g-lk} \rm For $\ell,k\geq 2$, we have $\Z(G(\ell,k))=\ell$. This is because any zero forcing set must contain at least $\ell-1$ degree-1 neighbors of $v$, but a set of this type is not a zero forcing set. However, the set of all degree-1 neighbors of $v$ is a minimum zero forcing set. 
\end{observation}

\begin{definition} \label{defn:star-path2} \rm Let $G^s(\ell,k)$ be $G(\ell,k)$ plus the vertices $u_1, u_2, \ldots, u_s$ each of whose only adjacency is to the neighbor of $v$ on the path, called $x$. We call the neighbor of $x$ on the long path $y$. The graph $G^3(\ell,k)$ is illustrated in Figure~$\ref{exa:star-path2}$.
\end{definition}

\begin{figure}[ht]
\centering
\begin{tikzpicture}[scale=.7]
\draw[thick] (0,1.75)--(2,3)--(6.3,3);
\draw[thick] (0,4.25)--(2,3);
\draw[thick] (0,1)--(2,3)--(0,5);
\draw[thick] (7.7,3)--(9,3);
\draw[thick] (2.3,1.5)--(3,3)--(3,1.5);
\draw[thick] (3,3)--(3.7,1.5);

\filldraw[black]
(0,1.75) circle [radius=3.5pt]
(0,1) circle [radius=3.5pt]
(0,5) circle [radius=3.5pt]
(0,4.25) circle [radius=3.5pt]
;
\filldraw[white]
(0,1.75) circle [radius=2.5pt]
(0,1) circle [radius=2.5pt]
(0,5) circle [radius=2.5pt]
(0,4.25) circle [radius=2.5pt]
;

\filldraw[black]
(2,3) circle [radius=3.5pt]
(3,3) circle [radius=3.5pt]
(4,3) circle [radius=3.5pt]
(5,3) circle [radius=3.5pt]
(6,3) circle [radius=3.5pt]

(8,3) circle [radius=3.5pt]
(9,3) circle [radius=3.5pt]
;
\filldraw[white]
(2,3) circle [radius=2.5pt]
(3,3) circle [radius=2.5pt]
(4,3) circle [radius=2.5pt]
(5,3) circle [radius=2.5pt]
(6,3) circle [radius=2.5pt]

(8,3) circle [radius=2.5pt]
(9,3) circle [radius=2.5pt]
;

\filldraw[black]
(0,3.3) circle [radius=1pt]
(0,3) circle [radius=1pt]
(0,2.7) circle [radius=1pt]
(6.7,3) circle [radius=1pt]
(7,3) circle [radius=1pt]
(7.3,3) circle [radius=1pt]
;

\node(0) at (-.5,3) {\Huge \color{cyan} $\Biggl\{$};
\node(1) at (-1.3,3) {\Large \color{cyan} $\ell$};
\node(2) at (6,3.9) {\Huge \color{cardinal} $\overbrace{\;\;\;\;\;\;\;\;\;\;\;\;\;\;\;\;\;}$};
\node(3) at (6,4.8) {\Large \color{cardinal} $k$};

\node(5) at (2.2,3.3) { $v$};
\node(6) at (3.2,3.3) {$x$};
\node(7) at (4.2,3.3) {$y$};

\filldraw[black]
(2.3,1.5) circle [radius=3.5pt]
(3,1.5) circle [radius=3.5pt]
(3.7,1.5) circle [radius=3.5pt]
;

\filldraw[white]
(2.3,1.5) circle [radius=2.5pt]
(3,1.5) circle [radius=2.5pt]
(3.7,1.5) circle [radius=2.5pt]
;

\node(9) at (2.3,1) {$u_1$};
\node(9) at (3,1) {$u_2$};
\node(9) at (3.7,1) {$u_3$};

\end{tikzpicture}

\caption{The graph $G^3(\ell,k)$.}
\label{exa:star-path2}
\end{figure}
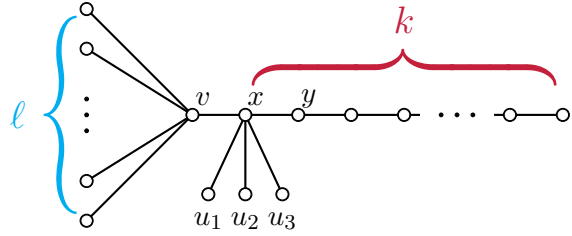

\begin{observation} \label{obs:g-lk2} \rm For $\ell,k\geq 2$, we have $\Z(G^s(\ell,k))=\ell+s-1$. This follows because at least $\ell-1$ degree-1 neighbors of $v$ and at least $s-1$  degree-1 neighbors of $x$ have to be included in any zero forcing set, but by themselves, they do not constitute a zero forcing set. However, a set consisting of these vertices  with the addition of the last leaf adjacent to  $x$ is a zero forcing set. \end{observation}

\subsection{Vertex Deletion}

Our next results show that the difference between the transmission zero forcing number of a graph $G$ and an induced subgraph on $|V(G)|-1$ vertices is unbounded from above and below. 

\begin{theorem}\label{thm:vert-del-1} 
For any $0<\alpha,\beta<1$ and positive integer $p$, there are values of $k,\ell$ such that  $\Zab(G(\ell,k)-v)-\Zab(G(\ell,k))\geq p$.
\end{theorem}

\begin{proof} Let $j$ be the smallest positive integer such that $\alpha^j<\beta$.  
Choose $k\geq 2$ large enough so that $\lceil \frac{2k}{2j+1}\rceil\geq p$. Choose $\ell\geq 2$ large enough so that $\ell \alpha^{k+1}\geq \beta$. 

First, we will show that $\Zab(G(\ell,k))= \ell$. By Observations~\ref{basic} (2) and \ref{obs:g-lk}, we have $\Zab(G(\ell,k))\geq \ell$. Let $B$ be the set of all degree-1 neighbors of $v$. Then at time step 1, $w(v)=\ell\alpha$. Zero forcing proceeds along the path, and the last vertex on the path receives a weight of $\ell\alpha^{k+1}$, which is greater than or equal to $\beta$. Hence $B$ is an $(\alpha, \beta)$-transmission forcing set, and $\Zab(G(\ell,k))= \ell$.

Next, we will show that $\Zab(G(\ell,k)-v)\geq \ell+p$. In $G(\ell,k)-v$, all of the degree-1 neighbors of $v$ become isolated vertices, and each must be included in any $(\alpha, \beta)$-transmission forcing set. The remaining vertices induce a path with $k$ vertices. By Theorem~\ref{thm:paths}, we have $\Zab(P_k) \geq \min\left\{\lceil \frac{k}{j}\rceil, \lceil \frac{2k}{2j+1}\rceil\right\} = \lceil \frac{2k}{2j+1}\rceil$. By hypothesis, $\lceil \frac{2k}{2j+1} \rceil\geq p$, so we have $\Zab(G(\ell,k)-v)\geq \ell+p$, as desired. 
\end{proof}

\begin{theorem} \label{thm:vert-del-2} 
For any $0<\alpha,\beta<1$ and positive integer $t$, there are values of $k,\ell,s$ such that  $\Zab(G^s(\ell,k))-\Zab(G^s(\ell,k)-u_s)\geq t$.
\end{theorem}

\begin{proof} Choose $s$ as the least positive integer so that $s\alpha\geq \beta$, then choose $i$ to be the least positive integer such that $(s+1)\alpha^i<\beta$.   
Let $j$ be the smallest positive integer such that $\alpha^j<\beta$.
Let $k$ be large enough so that 
$\lceil \frac{2(k-i)}{2j+1}\rceil\geq t+1$. Let $\ell$ be large enough so that $\ell\alpha \geq \beta$ and $\ell\alpha^{k+1}+(s-1)\alpha^k\geq \beta$. 

We observe that $G^s(\ell,k)-u_s \cong G^{s-1}(\ell,k)$. First, we will show that 
$\Zab(G^{s-1}(\ell,k))\leq \ell+s-1$.  
Let $B$ be the set of the degree-1 neighbors of $v$ and the degree-1 neighbors of $x$. Then $|B|=\ell+s-1$, and $w(v)=\ell\alpha$, $w(x)=(s-1)\alpha+\ell\alpha^2$, and the last vertex on the path has weight $(s-1)\alpha^{k} +\ell\alpha^{k+1}$, which is at least $\beta$ by our choice of $\ell$. Thus, $\Zab(G^{s-1}(\ell,k))\leq\ell+s-1$. 

Next, we will show that $\Zab(G^{s}(\ell,k))\geq\ell+s-1+t$. Let $B$ be an $(\alpha, \beta)$-transmission forcing set. Then $B$ contains at least $\ell-1$ degree-1 neighbors of $v$ and at least $s-1$ degree-1 neighbors of $x$. Let $P$ denote the (unique) induced path of $G^s(\ell,k)$, isomorphic to $P_{k+1}$, that contains $v,x$, and $y$ but no other neighbors of $v$ or $x$.

Suppose $x$ does not force $y$.  Then $w(y)\leq 1$ and the zero forcing process on $P-\{v,x\} \cong P_{k-1}$ proceeds as in Theorem~\ref{thm:paths}, which implies that at least $\lceil \frac{2(k-1)}{2j+1}\rceil\geq \lceil \frac{2(k-i)}{2j+1}\rceil\geq t+1$ more initially filled vertices are needed.  So $|B| \geq \ell-1+s-1+t+1=\ell+s-1+t$. 

Now assume $x$ forces $y$.  Then $\{u_1, \ldots, u_s\} \subseteq B$ and $y\notin B$.  We must consider two cases:  when $x\in B$ and when $x\notin B$.  Suppose $x\in B$. Then $w(x)=1$, and $P-v\cong P_{k}$. By Theorem~\ref{thm:paths}, we need at least $\lceil \frac{2k}{2j+1}\rceil\geq \lceil \frac{2(k-i)}{2j+1}\rceil\geq t+1$ more initially filled vertices, including $x$. 
It follows that $|B| \geq \ell-1 + s + t+1 = \ell + s + t$.  
Finally, suppose $x\notin B$.  Then $x$ receives a weight of $s\alpha\geq \beta$ from its degree-1 neighbors at time step 1 and becomes filled. If $x$ receives any weight from $v$ at time step 1, then $v\in B$; so $w(x) = s\alpha$ or $w(x) = (s+1)\alpha$. Then $x$ sends at most $(s+1)\alpha^2$ to $y$, and the transmission forcing process continues until the vertex on the path that receives weight at most $(s+1)\alpha^{i-1}$ and becomes filled. The remaining vertices induce a path with $k-i+1$ unfilled vertices.  
By Theorem~\ref{thm:paths}, the remaining $k-i$ vertices require at least $\lceil \frac{2(k-i)}{2j+1} \rceil \geq t+1$ additional initial vertices in $B$.  So $|B| \geq \ell-1+s+t+1=\ell+s+t$.  Note that 
the propagation from $x$ could be interrupted by a vertex in $B$; however, this would have the effect of replacing the path of unfilled vertices with an even longer path.  
Hence $\Zab(G^s(\ell,k))-\Zab(G^s(\ell,k)-u_s)\geq t$, as desired. 
\end{proof}

\subsection{Edge Deletion}
Next we show that deleting an edge can cause the transmission zero forcing number to go up or down arbitrarily. 

\begin{theorem} \label{thm:edge-del-a}
For any $0<\alpha,\beta<1$ and positive integer $p$, there are values of $k,\ell$ such that  $\Zab(G(\ell,k)-vx)-\Zab(G(\ell,k))\geq p$.
\end{theorem}

\begin{proof}
Let $j$ be the smallest exponent of $\alpha$ so that $\alpha^j<\beta$.  
Choose $k\geq 2$ large enough so that $\lceil \frac{2k}{2j+1}\rceil\geq p+1$. Choose $\ell\geq 2$ large enough so that $(\ell-1)\alpha^2\geq \beta$ and $\ell \alpha^{k+1}\geq \beta$. As in the proof of Theorem~\ref{thm:vert-del-1}, $\Zab(G(\ell,k))=\ell$. We will show that $\Z_{\alpha,\beta}(G(\ell,k)-vx)\geq \ell+p$. 
Now $G(\ell,k)-vx$ has two components, a star $K_{1,\ell}$ and a path $P_k$. Proposition~\ref{prop:star} and the choice of $\ell$ imply $\Zab(K_{1,\ell})=\ell-1$. Since $\Zab(P_k) \geq \lceil \frac{2k}{2j+1}\rceil$ by Theorem~\ref{thm:paths} and by hypothesis $\lceil \frac{2k}{2j+1}\rceil\geq p+1$, we have $\Zab(G(\ell,k)-vx)\geq (\ell-1)+p+1=\ell+p$, as desired. 
\end{proof}

\begin{theorem}
For any $0<\alpha,\beta<1$ and positive integer $t$, there are values of $k,\ell,s$ such that  $\Zab(G^s(\ell,k))-\Zab(G^s(\ell,k)-xu_s)\geq t$.
\end{theorem}
\begin{proof}
Let $t' = t+1$ and choose the corresponding values $k,\ell,s$ as in Theorem~\ref{thm:vert-del-2} so that $\Zab(G^s(\ell,k))-\Zab(G^s(\ell,k)-u_s) \geq t'$.  We observe that $G^s(\ell,k)-xu_s \cong G^{s-1}(\ell,k) \sqcup K_1 \cong G^s(\ell,k)-u_s \sqcup K_1$.  It follows by Observation~\ref{basic}(4) that $\Zab(G^s(\ell,k))-\Zab(G^s(\ell,k)-xu_s) \geq t'-1 = t$. 
\end{proof}

\subsection{Edge Contraction}

Similarly to vertex or edge deletion, edge contraction can cause the transmission zero forcing number to go up or down arbitrarily.  
\begin{corollary}
For any $0<\alpha,\beta<1$ and positive integer $t$, there are values of $k,\ell,s$ such that  $\Zab(G^s(\ell,k))-\Zab(G^s(\ell,k)/xu_s)\geq t$.   
\end{corollary}
\begin{proof}
Since $G^s(\ell,k)/xu_s \cong G^{s-1}(\ell,k) \cong G^s(\ell,k)-u_s$, this result is a direct consequence of Theorem~\ref{thm:vert-del-2}.
\end{proof}

\begin{theorem} \label{thm:edge-contr-a}
    For any $0 < \alpha, \beta < 1$ and positive integer $p$, there exists a graph $H$ and an edge $e\in E(H)$ such that $\Zab(H/e) - \Zab(H) \geq p$.
\end{theorem}

\begin{proof}
    Let $j$ be the smallest positive integer such that $\alpha^j < \beta$.  Choose $k\geq 2$ large enough so that $\lceil \frac{2k}{2j+1} \rceil \geq p+2$.  Choose $\ell \geq 2$ large enough so that $\ell\alpha^{k+1} \geq \beta$. Let $P$ denote the (unique) induced path on $k+1$ vertices of $G(\ell,k)$ that contains no neighbors of $v$ except $x$.  Let $H$ be the graph with vertex set $V(H) = V(G(\ell,k)) \cup \{z\}$ and edge set $E(H) = E(G(\ell,k)) \cup \{zw: w\in V(P)\}$; see Figure \ref{fig:edge-contr}.  Let $B$ be the set of all degree-1 neighbors of $v$ and the vertex $z$.  Then at time step 1, $w(v) = \ell\alpha$, and the transmission forcing process ends when the last vertex on the path receives a weight of $\ell\alpha^{k+1} \geq \beta$.  So $\Zab(H) \leq |B| = \ell+1$.

Let $e = vz$ and note that in the graph $H/e$, the identified vertex $v=z$ is adjacent to all other vertices; in particular, it is adjacent to every vertex in $P-v$.  An $(\alpha,\beta)$-transmission forcing set $B'$ must contain at least $\ell-1$ leaves adjacent to $v$. 
    
First, suppose $v$ never performs a force on a vertex in $P-v\cong P_k$. By Theorem~\ref{thm:paths}, this path requires at least $\lceil\frac{2k}{2j+1} \rceil \geq p+2$ initially filled vertices, so $|B'|\geq \ell-1+p+2=\ell+p+1$.
Now suppose $v$ performs a force to a vertex in $P-v$; note that $v$ can only do so after $k-1$ of its $k$ neighbors in $P-v$ have been filled. Furthermore, $v$ cannot force any of its degree-1 neighbors, hence all $\ell$ of them are in $B'$. Suppose that the vertex that $v$ forces is at distance $i+1$ from itself. Then $P-v \cong P_i \sqcup P_1 \sqcup P_{k-i-1}$.
By Theorem~\ref{thm:paths}, $B'$ must include at least $\lceil \frac{2i}{2j+1} \rceil + \lceil \frac{2(k-i-1)}{2j+1} \rceil \geq \lceil \frac{2(k-1)}{2j+1} \rceil \geq \lceil \frac{2k}{2j+1} \rceil - 1\geq p+1$ vertices of $P-v$.  Thus, we see that $|B'| \geq \ell + p+1$.  
We conclude that $\Zab(H/e) \geq p + \ell + 1$ and $\Zab(H/e) - \Zab(H) \geq p$.
\end{proof}

\begin{figure}[ht]
\centering
\begin{tikzpicture}[scale=.7]
\draw[thick] (0,1.75)--(2,3)--(6.3,3);
\draw[thick] (0,4.25)--(2,3);
\draw[thick] (0,1)--(2,3)--(0,5);
\draw[thick] (7.7,3)--(9,3);
\draw[thick] (5.2,1)--(2,3);
\draw[thick] (5.2,1)--(3,3);
\draw[thick] (5.2,1)--(4,3);
\draw[thick] (5.2,1)--(5,3);
\draw[thick] (5.2,1)--(6,3);
\draw[thick] (5.2,1)--(8,3);
\draw[thick] (5.2,1)--(9,3);

\filldraw[blue]

(0,1.75) circle [radius=3.5pt]
(0,1) circle [radius=3.5pt]
(0,5) circle [radius=3.5pt]
(0,4.25) circle [radius=3.5pt]
;

\filldraw[white]

(0,1.75) circle [radius=2.5pt]
(0,1) circle [radius=2.5pt]
(0,5) circle [radius=2.5pt]
(0,4.25) circle [radius=2.5pt]
;

\filldraw[black]
(2,3) circle [radius=3.5pt]
(3,3) circle [radius=3.5pt]
(4,3) circle [radius=3.5pt]
(5,3) circle [radius=3.5pt]
(6,3) circle [radius=3.5pt]

(8,3) circle [radius=3.5pt]
(9,3) circle [radius=3.5pt]
(5.2,1) circle [radius=3.5pt]

(0,3.3) circle [radius=.7pt]
(0,3) circle [radius=.7pt]
(0,2.7) circle [radius=.7pt]
(6.7,3) circle [radius=.7pt]
(7,3) circle [radius=.7pt]
(7.3,3) circle [radius=.7pt]
;

\filldraw[white]

(2,3) circle [radius=2.5pt]
(3,3) circle [radius=2.5pt]
(4,3) circle [radius=2.5pt]
(5,3) circle [radius=2.5pt]
(6,3) circle [radius=2.5pt]

(8,3) circle [radius=2.5pt]
(9,3) circle [radius=2.5pt]
(5.2,1) circle [radius=2.5pt]
;
\filldraw[black]
(0,3.3) circle [radius=1pt]
(0,3) circle [radius=1pt]
(0,2.7) circle [radius=1pt]
(6.7,3) circle [radius=1pt]
(7,3) circle [radius=1pt]
(7.3,3) circle [radius=1pt]
;

\node(0) at (-.5,3) {\Huge \color{cyan} $\Biggl\{$};
\node(1) at (-1.3,3) {\Large \color{cyan} $\ell$};
\node(2) at (6,3.6) {\Huge \color{cardinal} $\overbrace{\;\;\;\;\;\;\;\;\;\;\;\;\;\;\;\;\;\;}$};
\node(3) at (6,4.5) {\Large \color{cardinal} $k$};

\node(4) at (3.1,2.6) { \small \color{black} $x$};
\node(7) at (4,2.6) { \small \color{black} $y$};

\node(5) at (2.2,3.3) { $v$};

\node(6) at (5.4,.7){$z$};

\end{tikzpicture}

\caption{Illustration of the graph $H$ in Theorem \ref{thm:edge-contr-a}}
\label{fig:edge-contr}
\end{figure}
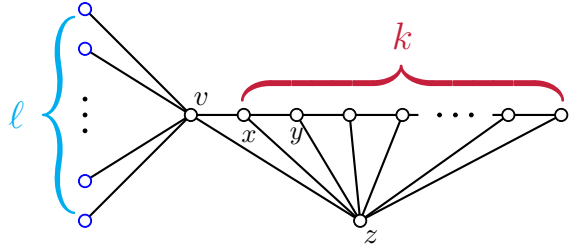

\section*{Acknowledgments}
The work of M.~Flagg is partially supported by NSF award 2331634.  The work of V.~Furst is partially supported by NSF award DMS-2331072.  The work of M.~Hunnell is partially supported by NSF award 2447261.  This project began at an American Institute of Mathematics (AIM) Workshop.  The authors are grateful to AIM and the NSF for their support.  We also thank Liam McAllister and Kana Takahashi for their careful reading and feedback.

\bibliographystyle{siam} 
\bibliography{Transmission.bib}

@article{edholm2012,
title = {Vertex and edge spread of zero forcing number, maximum nullity, and minimum rank of a graph},
journal = {Linear Algebra and its Applications},
volume = {436},
number = {12},
pages = {4352-4372},
year = {2012},
note = {Special Issue on Matrices Described by Patterns},
issn = {0024-3795},
doi = {https://doi.org/10.1016/j.laa.2010.10.015},
url = {https://www.sciencedirect.com/science/article/pii/S0024379510005392},
author = {Christina J. Edholm and Leslie Hogben and My Huynh and Joshua LaGrange and Darren D. Row}
}

@article{WU1979305,
title = {Cellular graph automata. I. basic concepts, graph property measurement, closure properties},
journal = {Information and Control},
volume = {42},
number = {3},
pages = {305-329},
year = {1979},
issn = {0019-9958},
doi = {https://doi.org/10.1016/S0019-9958(79)90288-2},
url = {https://www.sciencedirect.com/science/article/pii/S0019995879902882},
author = {Angela Wu and Azriel Rosenfeld}
}

@article{AMOS20151,
title = {Upper bounds on the k-forcing number of a graph},
journal = {Discrete Applied Mathematics},
volume = {181},
pages = {1-10},
year = {2015},
issn = {0166-218X},
doi = {https://doi.org/10.1016/j.dam.2014.08.029},
url = {https://www.sciencedirect.com/science/article/pii/S0166218X14003813},
author = {David Amos and Yair Caro and Randy Davila and Ryan Pepper}
}

@article{parameters2012,
author = {Barioli, Francesco and Barrett, Wayne and Fallat, Shaun M. and Hall, H. Tracy and Hogben, Leslie and Shader, Bryan and van den Driessche, P. and van der Holst, Hein},
title = {Parameters Related to Tree-Width, Zero Forcing, and Maximum Nullity of a Graph},
journal = {Journal of Graph Theory},
volume = {72},
number = {2},
pages = {146-177},
doi = {https://doi.org/10.1002/jgt.21637},
url = {https://onlinelibrary.wiley.com/doi/abs/10.1002/jgt.21637},
eprint = {https://onlinelibrary.wiley.com/doi/pdf/10.1002/jgt.21637},
year = {2013}
}

@phdthesis{owens,
	author = {K. Owens},
	school = {Brigham Young University - Provo},
	title = {Properties of the Zero Forcing Number},
	year = {2009}

}

@book {west1996intro,
    AUTHOR = {West, Douglas B.},
     TITLE = {Introduction to graph theory, second edition},
 PUBLISHER = {Prentice Hall, Inc., Upper Saddle River, NJ},
      YEAR = {2001},
     PAGES = {xvi+512},
      ISBN = {0-13-227828-6},
   MRCLASS = {05-01},
  MRNUMBER = {1367739},
}

@article{DREYER2009,
title = {Irreversible k-threshold processes: Graph-theoretical threshold models of the spread of disease and of opinion},
journal = {Discrete Applied Mathematics},
volume = {157},
number = {7},
pages = {1615-1627},
year = {2009},
issn = {0166-218X},
doi = {https://doi.org/10.1016/j.dam.2008.09.012},
url = {https://www.sciencedirect.com/science/article/pii/S0166218X08004393},
author = {Paul A. Dreyer and Fred S. Roberts}
}

@article{AIM2008,
title = {Zero forcing sets and the minimum rank of graphs},
journal = {Linear Algebra and its Applications},
volume = {428},
number = {7},
pages = {1628-1648},
year = {2008},
issn = {0024-3795},
doi = {https://doi.org/10.1016/j.laa.2007.10.009},
url = {https://www.sciencedirect.com/science/article/pii/S0024379507004624},
author = { {AIM Minimum Rank – Special Graphs Work Group}}
}

@article{BLESSING2015,
title = {On (t,r) broadcast domination numbers of grids},
journal = {Discrete Applied Mathematics},
volume = {187},
pages = {19-40},
year = {2015},
issn = {0166-218X},
doi = {https://doi.org/10.1016/j.dam.2015.02.005},
url = {https://www.sciencedirect.com/science/article/pii/S0166218X1500061X},
author = {David Blessing and Katie Johnson and Christie Mauretour and Erik Insko}
}

@book {domination1998,
    AUTHOR = {Haynes, Teresa W. and Hedetniemi, Stephen T. and Slater, Peter
              J.},
     TITLE = {Fundamentals of domination in graphs},
    SERIES = {Monographs and Textbooks in Pure and Applied Mathematics},
    VOLUME = {208},
 PUBLISHER = {Marcel Dekker, Inc., New York},
      YEAR = {1998},
     PAGES = {xii+446},
      ISBN = {0-8247-0033-3},
   MRCLASS = {05C69 (05-01)},
  MRNUMBER = {1605684},
MRREVIEWER = {Christina\ M.\ Mynhardt},
}

@article{quantum,
  title = {Full Control by Locally Induced Relaxation},
  author = {Burgarth, Daniel and Giovannetti, Vittorio},
  journal = {Phys. Rev. Lett.},
  volume = {99},
  issue = {10},
  pages = {100501},
  numpages = {4},
  year = {2007},
  month = {Sep},
  publisher = {American Physical Society},
  doi = {10.1103/PhysRevLett.99.100501},
  url = {https://link.aps.org/doi/10.1103/PhysRevLett.99.100501}
}

@article{Barioli2010ZeroFP,
  title={Zero forcing parameters and minimum rank problems},
  author={Francesco Barioli and Wayne W. Barrett and Shaun M. Fallat and H. Tracy Hall and Leslie Hogben and Bryan L. Shader and Pauline van den Driessche and Hein van der Holst},
  journal={Linear Algebra and its Applications},
  year={2010},
  volume={433},
  pages={401-411}
}

@article{article,
author = {Hogben, Leslie and Huynh, My and Kingsley, Nicole and Meyer, Sarah and Walker, Shanise and Young, Michael},
year = {2012},
month = {09},
pages = {1994–2005},
title = {Propagation time for zero forcing on a graph},
volume = {160},
journal = {Discrete Applied Mathematics},
doi = {10.1016/j.dam.2012.04.003}
}

@article{chilakamarri,
    author = {Chilakamarri, K. B. and Dean, N. and Kang, C. X. and Yi, E.},
    title = {Iteration index of a
          zero forcing set in a graph},
    journal = {Bulletin of the Institute of Combinatorics and Its
          Applications},
    year = {2012},
    volume = {64},
    pages = {57-72}
}
\end{document}